\documentclass{amsart}
\usepackage{amscd}
\usepackage{amsfonts}
\usepackage{amsmath, amsthm, amssymb}
\usepackage{latexsym}

\usepackage{graphicx}
\usepackage{color}
\usepackage{enumerate}
\usepackage{multicol}
\usepackage{array}
\usepackage[all]{xy}
\theoremstyle{plain} 
\newtheorem{thm}{Theorem}[section]

\newtheorem{lem}[thm]{Lemma}
\newtheorem{prop}[thm]{Proposition}

\theoremstyle{definition}
\newtheorem{defn}[thm]{Definition}
\newtheorem{exm}[thm]{Example}

\theoremstyle{remark}
\newtheorem{rem}[thm]{Remark}

\DeclareMathOperator{\GrMod}{GrMod}

\newcommand{\V}{\mathcal{V}}
\newcommand{\A}{\mathcal{A}(E,\sigma)}

\newcommand{\kang}{k\langle x, y, z \rangle}

\newcommand{\al}{\alpha}
\newcommand{\be}{\beta}
\newcommand{\ga}{\gamma}

\newcommand{\la}{\lambda}
\newcommand{\ep}{\varepsilon}

\makeatletter
\@namedef{subjclassname@2010}{%
  \textup{2010} Mathematics Subject Classification}
\makeatother
\title[DEF. REL. OF $3$-DIM. QUAD. AS-REGULAR ALGEBRAS]
{Defining relations of $3$-dimensional quadratic AS-regular algebras}
\author[A. Itaba]{Ayako Itaba}
\address{
Department of Mathematics, 
Faculty of Science, Tokyo University of Science\\
1-3 Kagurazaka, Shinjyuku, Tokyo, 162-8601, Japan}
\email{itaba@rs.tus.ac.jp}

\author[M. Matsuno]{Masaki Matsuno}
\address{
Graduate school of
Integrated Science and Technology, Shizuoka University\\
Ohya 836, Shizuoka 422-8529, Japan}
\email{matsuno.masaki.14@shizuoka.ac.jp}
\begin{document}
\begin{abstract}
Classification of AS-regular algebras is 
one of the main interests in non-commutative algebraic geometry. 
Recently, a complete list of superpotentials (defining relations) 
of all $3$-dimensional AS-regular algebras which are 
Calabi-Yau was given by Mori-Smith (the quadratic case) 
and Mori-Ueyama (the cubic case), 
however, no complete list of defining relations of 
all $3$-dimensional AS-regular algebras has not appeared 
in the literature. 
In this paper, we give all possible defining relations of
$3$-dimensional quadratic AS-regular algebras.  
Moreover, we classify them up 
to isomorphism and up to graded Morita equivalence in terms of their 
defining relations in the case that their point schemes are not elliptic 
curves.  In the case that their point schemes are elliptic curves, 
we give conditions for isomorphism and graded Morita equivalence 
in terms of geometric data.
\end{abstract}
\maketitle
\section{Introduction}
Classification of Artin-Schelter regular (AS-regular) algebras 
is one of the main interests in noncommutative algebraic geometry. 
It was originally defined by Artin-Schelter \cite{AS}, 
and in that paper, it was attempted 
to classify $3$-dimensional AS-regular algebras generated in degree $1$, 
partially using computer programs. 
It was shown in \cite{AS} that 
every $3$-dimensional AS-regular algebra generated in degree $1$ 
has either $3$ generators and $3$ quadratic defining relations 
(the quadratic case), 
or $2$ generators and $2$ cubic defining relations 
(the cubic case). 
In each case, a list of defining relations 
(in fact potentials in the modern terminology) of 
\lq\lq generic\rq\rq $3$-dimensional AS-regular algebras was given 
in \cite[Table (3.11)]{AS} (the quadratic case) 
and in \cite[Table (3.9)]{AS} (the cubic case). 
Soon after, Artin-Tate-Van den Bergh \cite{ATV1} found a nice one-to-one 
correspondence between the set of $3$-dimensional AS-regular algebras $A$ 
and the set of regular geometric pairs $(E,\sigma)$ 
where $E$ is a scheme 
and $\sigma \in {\rm Aut}_{k}E$, 
so the classification of $3$-dimensional AS-regular algebras 
reduces to the classification of regular geometric pairs. 
A list of regular geometric pairs corresponding to 
``generic'' $3$-dimensional AS-regular algebras 
was given in \cite[4.13]{ATV1}. 
(A complete list of regular geometric pairs 
``up to graded Morita equivalence''
in the quadratic case was given in \cite[Table 1]{BP}. See Remark \ref{Z-algs}.) 
This work convinced us that algebraic geometry is very useful 
to study even noncommutative algebras, 
and is considered as a starting point of the research field 
noncommutative algebraic geometry. 

Although the next natural project is to classify $4$-dimensional 
AS-regular algebras, which has been in fact very active until now, 
some ``non-generic'' $3$-dimensional AS-regular algebras 
were also studied (\cite{MU}, \cite{NVZ}, etc.). 
Recently, a complete list of superpotentials (defining relations) 
of all $3$-dimensional AS-regular algebras which are 
``Calabi-Yau'' was given in \cite{MS} (the quadratic case) 
and in \cite{MU2} (the cubic case), 
however, no complete list of defining relations of 
``all'' $3$-dimensional AS-regular algebras has not appeared 
in the literature. 
So the goal of our project is 
\begin{enumerate}[(I)]
\item to give a complete list of defining relations of 
      ``all'' $3$-dimensional quadratic AS-regular algebras, 
\item to classify them up to isomorphism 
      in terms of their defining relations, and 
\item to classify them up to graded Morita equivalence 
      in terms of their defining relations. 
\end{enumerate}
In this paper, we completed our project 
in the case that the point scheme is not an elliptic curve. 

This paper is organized as follows: 
In Section \ref{sec_preliminary}, 
we recall the definitions of a twisted algebra from \cite{Z},
a geometric algebra from \cite{Mo1}, 
and an AS-regular algebra from \cite{AS}. 
In Section \ref{sec_main}, 
we give a complete list of defining relations of 
$3$-dimensional quadratic AS-regular algebras
whose point schemes are not elliptic curves,
and classify them up to isomorphism and up to graded Morita equivalence
in terms of their defining relations 
(see Theorems \ref{thm_main}, \ref{thm_main2}).
In particular, in the case that the point scheme is a nodal cubic curve,
we found a new algebra which is not isomorphic to any algebra
classified in \cite{NVZ} (see Remark \ref{rem_typeNC}).
Finally, in Section \ref{sec_EC}, 
we give a complete list of defining relations of 
geometric algebras
whose point schemes are elliptic curves
(which include $3$-dimensional quadratic AS-regular algebras
whose point schemes are elliptic curves),
and 
conditions for isomorphism and graded Morita equivalence
in terms of geometric data 
(see Theorems \ref{df}, \ref{classify}, \ref{main4}).
\section{Preliminary}
\label{sec_preliminary}
Throughout this paper, 
we fix an algebraically closed field $k$ 
of characteristic zero, and assume that
a graded $k$-algebra is
an $\mathbb{N}$-graded algebra 
$A=\bigoplus_{i \in \mathbb{N}}A_{i}$. 
A {\it connected graded} algebra is a graded algebra 
$A=\bigoplus_{i \in \mathbb{N}}A_{i}$ such that $A_{0}=k$. 
We denote by $\GrMod A$ 
the category of graded right $A$-modules. 
Morphisms in $\GrMod A$ are right $A$-module homomorphisms 
preserving degrees. 
We say that two graded algebras $A$ and $A'$ 
are {\it graded Morita equivalent} 
if the categories $\GrMod A $ and $ \GrMod A'$ 
are equivalent. 
\subsection{Twisted Algebras}

For a graded algebra $A$, 
Zhang \cite{Z}
introduced a notion of
twisted algebra $A^{\varphi}$ of $A$ 
by a graded algebra automorphism $\varphi \in {\rm GrAut}_{k}A$. 
In this paper, we only define a twisted algebra 
for a quadratic algebra. 
A quadratic algebra $A$ is of the form $T(V)/(R)$ 
where $V$ is a finite-dimensional $k$-vector space, 
$T(V)$ is the tensor algebra of $V$, $R\subset V\otimes_{k}V$ is a subspace
and $(R)$ is the two-sided ideal of $T(V)$ generated by $R$. 
We denote the general linear group of $V$ by ${\rm GL}\,(V)$.
It is easy to check the following lemma.
\begin{lem}\label{lem1}
	Let $A=T(V)/(R)$ and $A'=T(V)/(R')$ be quadratic algebras with $R,R' \subset V \otimes_{k} V$.
	Then $A \cong A'$ if and only if there is $\phi \in {\rm GL}(V)$ such that $R'=(\phi \otimes \phi)(R)$.
\end{lem}
\begin{defn}\label{twa}
	Let $V$ be a finite-dimensional $k$-vector space
	and $A=T(V)/(R)$ a quadratic algebra with
	$R \subset V \otimes_{k} V$.
	\begin{enumerate}[{\rm (1)}]
		\item For $\phi \in {\rm GL}(V)$,
		we define the twisted algebra $A^{\phi}:=T(V)/(R^{\phi})$ of $A$ by $\phi$
		where $R^{\phi}:=(\phi \otimes {\rm id})(R)\subset V\otimes_{k} V$.
		\item For $\varphi \in {\rm GrAut}_{k}A$,
		we define the twisted algebra $A^{\varphi}:=A^{\varphi|_{V}}$ of $A$
		by $\varphi$ where $\varphi|_{V} \in {\rm GL}(V)$.
	\end{enumerate}
\end{defn}
For a quadratic algebra $A$ and $\phi \in {\rm GL}(V)$, 
it follows from the definition
that $(A^{\phi})^{\phi^{-1}}=A$. 
If $\varphi \in {\rm GrAut}_{k}A$, 
then 
$\varphi \in {\rm GrAut}_{k}A^{\varphi}$ and $(A^{\varphi})^{\varphi^{-1}}=A$. 
Since $A^{\varphi}$ is isomorphic to the twisted algebra
defined in \cite{Z}, the following theorem is shown.

\begin{thm}[{\rm \cite[Theorem 3.1]{Z}}]\label{twist}
	Let $V$ be a finite-dimensional $k$-vector space and 
	$A=T(V)/(R)$ a quadratic algebra with $R\subset V\otimes V$. 
	If $\varphi \in {\rm GrAut}_{k}A$, 
	then ${\rm GrMod}\,A \cong {\rm GrMod}\,A^{\varphi}$.
\end{thm}

\begin{rem}\label{ta}
	Let $A=T(V)/(R)$ be a quadratic algebra and $\phi \in {\rm GL}(V)$.
	If $(\phi \otimes \phi)(R)=R$, then $\phi$ extends to $\overline{\phi} \in {\rm GrAut}_{k}A$,
	so ${\rm GrMod}\,A \cong {\rm GrMod}\,A^{\overline{\phi}}={\rm GrMod}\,A^{\phi}$.
	However, when $(\phi \otimes \phi)(R) \neq R$, 
        $A$ may not be graded Morita equivalent to $A^{\phi}$ 
	(See Example \ref{exam2}). 
\end{rem}
\subsection{Geometric Algebras}
Let $V$ be a finite dimensional $k$-vector space.
The equivalence relation on $V \setminus \{ 0 \}$ is defined by 
$$
u \thicksim v \,\Longleftrightarrow\, 
{\rm there \,\,exists}\,\, \la \in k^{\times}\,\, {\rm with}\,\, u=\la v.
$$
The projective space associated to $V$ is defined by
$$
\mathbb{P}(V):=V \setminus \{0\}/\thicksim.
$$
For $\phi \in {\rm GL}(V)$,
the map $\overline{\phi^{\ast}}:\mathbb{P}(V^{\ast}) \rightarrow \mathbb{P}(V^{\ast})$
defined by $\overline{\phi^{\ast}}(\overline{\xi})=\overline{\phi^{\ast}(\xi)}$ 
is an automorphism 
where $\phi^{\ast}:V^{\ast} \rightarrow V^{\ast}$ 
is the dual map of $\phi$.
For $\phi,\psi \in {\rm GL}(V)$,
the map $\phi \times \psi:V \times V \rightarrow V \otimes_{k} V$ defined by $(\phi \times \psi)(v,w)=\phi(v) \otimes \psi(w)$
is a bilinear map and induces a linear map $\phi \otimes \psi:V \otimes_{k} V \rightarrow V \otimes_{k} V$ by
$(\phi \otimes \psi)(v \otimes w)=\phi(v) \otimes \psi(w)$ where $v$, $w \in V$.
For $g=\sum v_{i}\otimes w_{i} \in V \otimes_{k} V$,
we write
$$
g(p,q)=\sum \xi(v_{i})\eta(w_{i})
$$
where $p=\overline{\xi},q=\overline{\eta} \in \mathbb{P}(V^{\ast})$.
Note that the zero set of $R \subset V \otimes_{k} V$,
$$
\mathcal{V}(R):=\{ (p,q)\in\mathbb{P}(V^{\ast})\times\mathbb{P}(V^{\ast})\,|\,g(p,q)=0\,\,{\rm for\,\,any\,\,} g \in R \}
$$
is well-defined.

In \cite{Mo1}, the notion of {\it geometric algebra} was introduced. 
\begin{defn}[\cite{Mo1}]
	\label{geometric}
	A geometric pair $(E,\sigma)$ consists of a projective variety 
        $E \subset \mathbb{P}(V^{\ast})$
	and $\sigma \in {\rm Aut}_{k}\,E$.
	Let $A=T(V)/(R)$ be a quadratic algebra with $R \subset V \otimes_{k} V$.
	\begin{enumerate}[{\rm (1)}]
		\item We say that $A$ satisfies (G1) if there exists a geometric pair $(E,\sigma)$ such that
		$$
		\mathcal{V}(R)=\{ (p,\sigma(p)) \in \mathbb{P}(V^{\ast})\times\mathbb{P}(V^{\ast})\,|\,p \in E \}.
		$$
		In this case, we write $\mathcal{P}(A)=(E,\sigma)$, 
                and call $E$ the {\it point scheme} of $A$. 
		\item We say that $A$ satisfies (G2) if there exists a geometric pair $(E,\sigma)$ such that
		$$
		R=\{ f \in V \otimes_{k} V\,|\,f(p,\sigma(p))=0\,\,{\rm for\,\,any\,\,} p \in E \}.
		$$
		In this case, we write $A=\mathcal{A}(E,\sigma)$.
		\item A quadratic algebra $A$ is called geometric if $A$ satisfies both (G1) and (G2)
		with $A=\mathcal{A}(\mathcal{P}(A))$.
	\end{enumerate}
\end{defn}
If $A$ satisfies (G1), then $A$ determines the pair $(E,\sigma)$.
Conversely, if $A$ satisfies (G2), then $A$ is determined by the pair $(E,\sigma)$.
When we say that $\mathcal{A}(E,\sigma)$ is geometric,
we tacitly assume that $\mathcal{P}(\mathcal{A}(E,\sigma))=(E,\sigma)$
so that the point scheme of $\mathcal{A}(E,\sigma)$ is $E$.

Note that, for $g=\sum v_{i} \otimes w_{i} \in V \otimes_{k} V$,
$\phi,\psi \in {\rm GL}\,(V)$ and $p,q \in \mathbb{P}(V^{\ast})$,
$((\phi \otimes \psi)(g))(p,q)=0$ 
if and only if $g\left(\overline{\phi^{\ast}}(p),\overline{\psi^{\ast}}(q)\right)=0$.

\begin{prop}\label{ge}
	Let $E \subset \mathbb{P}(V^{\ast})$ be a projective variety,
	$\sigma \in {\rm Aut}_{k}\,E$ and $\phi \in {\rm GL}\,(V)$.
	Suppose that $\overline{\phi^{\ast}} \in {\rm Aut}_{k}\mathbb{P}(V^{\ast})$
	restricts to $\overline{\phi^{\ast}} \in {\rm Aut}_{k}\,E$.
	Let $A=T(V)/(R)$ be a quadratic algebra with $R \subset V \otimes_{k} V$.
	\begin{enumerate}[{\rm (1)}]
		\item 
		$\mathcal{A}\left( E,\sigma\overline{\phi^{\ast}} \right)
                =\mathcal{A}(E,\sigma)^{\phi}$.
		
		\item If $\mathcal{P}(A)=(E,\sigma)$,
		then $\mathcal{P}(A^{\phi})=(E,\sigma\overline{\phi^{\ast}})$.
		
		\item If $A$ is geometric with $\mathcal{P}(A)=(E,\sigma)$,
		then $A^{\phi}$ is geometric with $\mathcal{P}(A^{\phi})
                =(E,\sigma\overline{\phi^{\ast}})$.
	\end{enumerate}
\end{prop}
\begin{proof}
	\begin{enumerate}[{\rm (1)}]
		\item By (G2), we can write
		$\mathcal{A}(E,\sigma)=T(V)/(R_{1})$ where
		$$
                R_{1}=\{ f \in V \otimes_{k} V\,|\,f(p,\sigma(p))=0\,\,{\rm for\,\,any\,\,} p \in E  \},
                $$
		and $\mathcal{A}(E,\sigma\overline{\phi^{\ast}})=T(V)/(R_{2})$ where
		$$
                R_{2}=\{ f \in V \otimes_{k} V\,|\,f(p,\sigma\overline{\phi^{\ast}}(p))=0\,\,{\rm for\,\,any\,\,} p \in E  \}.
                $$
		Since $\overline{\phi^{\ast}} \in {\rm Aut}_{k}\,E$,
		\begin{align*}
		f \in R_{2}
		&\Longleftrightarrow f(p,\sigma\overline{\phi^{\ast}}(p))=0\,\,{\rm for\,\,any\,\,} p \in E \\
		&\Longleftrightarrow f\left(\left(\overline{\phi^{\ast}}\right)^{-1}(p),\sigma(p)\right )=0
		\,\,{\rm for\,\,any\,\,} p \in E \\
		&\Longleftrightarrow \left( (\phi^{-1} \otimes {\rm id})(f)  \right)(p,\sigma(p))=0\,\,{\rm for\,\,any\,\,} p \in E \\
		&\Longleftrightarrow (\phi^{-1} \otimes {\rm id})(f) \in R_{1} \\
		&\Longleftrightarrow f \in (\phi \otimes {\rm id})(R_{1})=:R_{1}^{\phi},
		\end{align*}
                so 
		$R_{2}=R_{1}^{\phi}$.
		
		\item  Suppose that $\mathcal{P}(A)=(E,\sigma)$, that is,
		$\mathcal{V}(R)=\{ (p,\sigma(p)) \in \mathbb{P}(V^{\ast})\times\mathbb{P}(V^{\ast})\,|\,p \in E \}$.
		Since $\overline{\phi^{\ast}} \in {\rm Aut}_{k}\,E$, 
		\begin{align*}
		(p,q) \in \mathcal{V}(R^{\phi})
		&\Longleftrightarrow g(p,q)=0\,\,{\rm for\,\,any\,\,} g \in R^{\phi} \\
		&\Longleftrightarrow \left( (\phi \otimes {\rm id})(f) \right)(p,q)=0\,\,{\rm for\,\,any\,\,} f \in R \\
		&\Longleftrightarrow f \left( \overline{\phi^{\ast}}(p),q \right)=0\,\,{\rm for\,\,any\,\,} f \in R \\
		&\Longleftrightarrow \left( \overline{\phi^{\ast}}(p),q \right) \in \mathcal{V}(R) \\
		&\Longleftrightarrow q=\sigma\overline{\phi^{\ast}}(p), p \in E \\
		&\Longleftrightarrow (p,q) \in
		\{ (p,\sigma\overline{\phi^{\ast}}(p))\in \mathbb{P}(V^{\ast})\times\mathbb{P}(V^{\ast})\,|\,p \in E \},
		\end{align*}
		so
		$\mathcal{P}(A^{\phi})=(E,\sigma\overline{\phi^{\ast}})$.
		
		\item 	Suppose that $A$ is geometric with $\mathcal{P}(A)=(E,\sigma)$.
		Since $\mathcal{P}(A)=(E,\sigma)$, $\mathcal{P}(A^{\phi})=(E,\sigma\overline{\phi^{\ast}})$ by ($2$).
		Since $A=\mathcal{A}(\mathcal{P}(A))=\mathcal{A}(E,\sigma)$,
		$\mathcal{A}(\mathcal{P}(A^{\phi}))=\mathcal{A}(E,\sigma\overline{\phi^{\ast}})=\mathcal{A}(E,\sigma)^{\phi}
		=A^{\phi}$ by ($1$).
	\end{enumerate}
\end{proof}
\begin{defn}
	Let $X,Y \subset \mathbb{P}(V)$ be two projective varieties.
	We say that $X$ and $Y$ are \textit{projectively equivalent}
	if there exists an isomorphism $\phi:X \rightarrow Y$ 
        which extends to an automorphism of $\mathbb{P}(V)$. 
	We call $\phi$ a {\it projective equivalence} from $X$ to $Y$.
\end{defn}

The following theorem tells us that 
classifying geometric algebras is equivalent to 
classifying geometric pairs. 
\begin{thm}[{{\cite[Remark 4.9]{Mo1}, cf. \cite{ATV1}}}]
	\label{gradediso}
	Let $A=\mathcal{A}(E,\sigma)$ and $A'=\mathcal{A}(E',\sigma')$ 
        be geometric algebras.
	Then $A$ is isomorphic to $A'$ as graded $k$-algebras 
	if and only if there is a projective equivalence $\phi$
	from $E$ to $E'$, such that the following diagram commutes{\rm : }
	$$
	\begin{CD}
	E @>\phi >> E'\\
	@V\sigma VV @VV \sigma' V\\
	E @>\phi >> E'
	\end{CD} $$
\end{thm}
\begin{thm}[{\cite[Theorem 4.7]{Mo1}}]
	\label{gradedequi}
	Let $A=\mathcal{A}(E,\sigma)$ and $A'=\mathcal{A}(E',\sigma')$ 
        be geometric algebras.
	Then ${\rm GrMod}\,A \cong {\rm GrMod}\,A'$ 
	if and only if there exists a sequence $\{ \phi_{i}\}_{i \in \mathbb{Z}}$
	of projective equivalences from $E$ to $E'$ 
	such that the following diagram commute for all  $i \in \mathbb{Z}${\rm : }
	$$
	\begin{CD}
	E @>\phi_{i} >> E'\\
	@V\sigma VV @VV \sigma ' V\\
	E @>\phi_{i+1} >> E'
	\end{CD}$$
\end{thm}
\subsection{AS-regular algebras}
Artin and Schelter \cite{AS} defined a class of regular algebras
which are main objects of study in noncommutative algebraic geometry.
\begin{defn}[\cite{AS}]
	A connected graded algebra $A$ is called 
	a {\it $d$-dimensional Artin-Schelter regular 
	{\rm (simply} AS-regular{\rm )} algebra} 
	if $A$ satisfies the following conditions: 
	\begin{enumerate}[(i)]
		\item ${\rm gldim}\,A=d<\infty$, 
		\item ${\rm GKdim}\,A:={\rm inf}\{\alpha\in \mathbb{R} \mid {\rm dim}_{k}(\sum_{i=0}^{n}A_{i})\leq n^{\alpha}\text{ for all }n \gg 0\}<\infty$, and,
		\item ({\it Gorenstein condition})\quad 
		${\rm Ext}_{A}^{i}(k,A)=\left\{
		\begin{array}{ll}
		k   &\quad (i=d), \\
		0   &\quad (i\neq d). 
		\end{array}
		\right.$
	\end{enumerate}
\end{defn}
A $3$-dimensional AS-regular algebra $A$ finitely generated in degree $1$
is one of the following forms: 
\[
A=k\langle x, y, z \rangle/(f_{1},f_{2},f_{3})
\]
where $f_{i}$ are homogeneous polynomials of degree $2$ (the quadratic case), or
\[
A=k\langle x, y\rangle/(g_{1}, g_{2})
\]
where $g_{i}$ are homogeneous polynomials of degree $3$ (the cubic case) 
(see {\cite[Theorem 1.5]{AS}}). 
Our main focus of this paper is to study $3$-dimensional quadratic AS-regular algebras. 
\begin{thm}[\cite{ATV1}]
	\label{thmATV1}
	Every $3$-dimensional quadratic AS-regular algebra $A$ 
	is geometric. 
	Moreover, the point scheme $E$ of $A$ is either $\mathbb{P}^{2}$ or 
	a cubic divisor in $\mathbb{P}^{2}$. 
\end{thm}
\begin{rem}
	In the above theorem, $E\subset \mathbb{P}^{2}$ could be a 
	 non-reduced cubic divisor in $\mathbb{P}^{2}$.
	See \cite[Definition 4.3]{Mo1} for the definition of a geometric algebra 
	in the case that $E$ is non-reduced. 
\end{rem}
We call a geometric pair $(E,\sigma)$ {\it regular} 
if $(E,\sigma)=\mathcal{P}(A)$
for some $3$-dimensional quadratic AS-regular 
algebra $A$. 
The above theorem shows that
the classification of $3$-dimensional quadratic AS-regular algebras reduces 
to the classification of regular geometric pairs.

The types of regular geometric pairs are defined in \cite{MU} 
which are slightly modified from the original types 
defined in \cite{AS} and \cite{ATV1}. 
We extend the types defined in \cite{MU} as follows
(since ${\rm Aut}_{k}\,\mathbb{P}^{n-1} \cong {\rm PGL}_{n}(k)$, 
we often identify $\sigma \in {\rm Aut}_{k}\mathbb{P}^{n-1}$ 
with the representing matrix 
$\sigma \in {\rm PGL}_{n}(k)$): 
\begin{description}
	\item[{\rm (1) Type P}] 
	$E$ is $\mathbb{P}^{2}$, 
	and $\sigma \in \mathrm{Aut}_{k}\mathbb{P}^{2} \cong \mathrm{PGL}_{3}(k)$ 
	(Type P is divided into 
	Type P$_{i}$ ($i=1,2,3$) in terms of the Jordan canonical form of $\sigma$). 
	\item [{\rm (2-1) Type S$_{1}$}] 
	$E$ is a triangle, and $\sigma$ stabilizes each component. 
	\item[{\rm (2-2) Type S$_{2}$}]
	$E$ is a triangle, and $\sigma$ interchanges two of its components. 
	\item[{\rm (2-3) Type S$_{3}$}]
	$E$ is a triangle, and $\sigma$ circulates three components. 
	\item[{\rm (3-1) Type S$'_{1}$}] 
	$E$ is a union of a line and a conic meeting at two points, 
	and $\sigma$ stabilizes each component 
	and two intersection points.
	\item[{\rm (3-2) Type S$'_{2}$}]
	$E$ is a union of a line and a conic meeting at two points, 
	and $\sigma$ stabilizes each component 
	and interchanges two intersection points.
	\item[{\rm (4-1) Type T$_{1}$}]
	$E$ is a union of three lines meeting at one point, 
	and $\sigma$ stabilizes each component.
	\item[{\rm (4-2) Type T$_{2}$}]
	$E$ is a union of three lines meeting at one point, 
	and $\sigma$ interchanges two of its components. 
	\item[{\rm (4-3) Type T$_{3}$}]
	$E$ is a union of three lines meeting at one point, 
	and $\sigma$ circulates three components.
	\item[{\rm (5) Type T$'$}]
	$E$ is a union of a line and a conic meeting at one point, 
	and $\sigma$ stabilizes each component. 
	\item[{\rm (6) Type CC}] $E$ is a cuspidal cubic curve.  
	\item[{\rm (7) Type NC}]
	$E$ is a nodal cubic curve 
	(Type NC is divided into Type NC$_{i}$ $(i=1,2)$). 
	\item[{\rm (8) Type WL}]$E$ is a union of a double line and a line 
	(Type WL is divided into Type WL$_{i}$ $(i=1,2,3)$). 
	\item[{\rm (9) Type TL}] $E$ is a triple line 
	(Type TL is divided into Type TL$_{i}$ $(i=1,2,3,4)$). 
	\item[{\rm (10) Type EC}] $E$ is an elliptic curve. 
\end{description}

\begin{exm}[{\cite[Example 4.10]{Mo1}}]
	\label{exm_Mo1}
	$3$-dimensional quadratic AS-regular algebras 
        $A=\mathcal{A}(E,\sigma)$ of Type $S_{1}$
	are classified by the following steps.
		 \item[{\rm Step 0:}] Since $E$ is a union of three lines making a triangle,
		$E$ is projectively equivalent to
		$\mathcal{V}(xyz)=\mathcal{V}(x)\cup\mathcal{V}(y)\cup\mathcal{V}(z)$, so
		we may assume that $E=\mathcal{V}(xyz)=\mathcal{V}(x)\cup\mathcal{V}(y)\cup\mathcal{V}(z)$
		by Theorem \ref{gradediso}.
		\item[{\rm Step 1:}] Since $\sigma\in {\rm Aut}_{k}E$ stabilizes each component,
		$\sigma \in {\rm Aut}_{k}E$ is given by
		\begin{align*}
		\sigma|_{\mathcal{V}(x)}(0:b:c)&=(0:b:\al c), \\
		\sigma|_{\mathcal{V}(y)}(a:0:c)&=(\be a:0:c), \\
		\sigma|_{\mathcal{V}(z)}(a:b:0)&=(a:\ga b:0),
		\end{align*}
		where $\al,\be,\ga \in k$ and $\al \be \ga \neq 0, 1$.
		\item[{\rm Step 2:}] By using (G2) condition in Definition \ref{geometric},
		we can compute the defining relation of $A=\mathcal{A}(E,\sigma)$ as
		$$yz-\al zy,\,\,zx - \be xz,\,\,xy-\ga yx.$$ \\
		Let $A'$ be another algebra of Type $S_{1}$ 
		with the defining relations 
		$$
		yz - \al' zy,\,zx-\be'xz,\,xy-\ga'yx, 
		$$
		where $\al',\be',\ga' \in k$ and $\al' \be' \ga' \neq 0, 1$.
		\item[{\rm Step 3:}] By Theorem \ref{gradediso}, we can show that
		$A\cong A'$ as graded $k$-algebras 
		if and only if 
		$$
		(\al',\be',\ga')=
		\begin{cases}
		(\al,\be,\ga),\,
		(\be,\ga,\al),\,
		(\ga,\al,\be),\\
		(\al^{-1},\ga^{-1},\be^{-1}),\,
		(\be^{-1},\al^{-1},\ga^{-1}),\,
		(\ga^{-1},\be^{-1},\al^{-1}).
		\end{cases}
		$$ 
		\item[{\rm Step 4:}] By Theorem \ref{gradedequi}, we can show that
		$\mathrm{GrMod} A \cong  \mathrm{GrMod} A'$ 
		if and only if 
		$\al'\be'\ga'=(\al\be\ga)^{\pm 1}$. 
	
\end{exm}
The purpose of this paper 
is to expand the above example to the remaining types. 
\section{Defining relations for non Type EC algebras}
\label{sec_main}
The following theorem lists all possible defining relations of algebras 
in each type up to isomorphism except for Type EC.
\begin{thm} [\cite{E,K,KMM,Ma,O}]
\label{thm_main}
Let $A=\A$ be a $3$-dimensional quadratic AS-regular algebra. 
For each type except for Type EC, the following table describes 
\begin{description}
\item[{\rm (I)}] the defining relations of $A$, and
\item[{\rm (II)}] the conditions to be isomorphic as graded algebras 
      in terms of their defining relations. 
      {\rm (}see Example \ref{exm_Mo1}.{\rm )}
\end{description}
In the following table, 
if $X \neq Y$ or $i \neq j$, then Type X$_{i}$ algebra 
is not isomorphic to any Type Y$_{j}$ algebra.
\end{thm}

\noindent
{\renewcommand\arraystretch{1.5} 
\begin{tabular}{|p{0.7cm}|p{4.3cm}|p{6.3cm}|} \hline
 {\footnotesize Type}  &\quad (I) defining relations
       &\quad (II) condition to be  \\ 
 &\quad ($\alpha,\beta,\gamma \in k$) 
       &\quad graded algebra isomorphic \\ 
            \hline
\end{tabular}}

\noindent{\renewcommand\arraystretch{1.5} 
\begin{tabular}{|p{0.7cm}|p{4.3cm}|p{6.3cm}|} \hline
       P$_{1}$   
            &
            $
            \begin{cases}
            \al xy - \be yx ,\\
            \be yz- \gamma zy,\\
            \gamma zx-\alpha xz
            \end{cases}\hspace{-10pt}
            $
            ($\alpha\beta\gamma \neq 0$)
            &\quad
\begin{minipage}{6cm}
            \rule{0pt}{12pt}$
            (\alpha':\beta':\gamma')$
            \\
            $=
            \begin{cases}
            (\alpha:\beta:\gamma),\,(\alpha:\gamma:\beta),\\
            (\beta:\alpha:\gamma),\,(\beta:\gamma:\alpha),\\
            (\gamma:\alpha:\beta),(\gamma:\beta:\alpha)
            \end{cases}
            \text{ in }\mathbb{P}^{2}
            $
\end{minipage}
            \\ \hline
     P$_{2}$  &
            $
            \begin{cases}
            xy-yx+y^{2},\\
            xz-\alpha zx+\alpha zy,\\
            yz-\alpha zy \ \ \ \ \ \ \ \text{($\alpha \neq 0$)}
            \end{cases}
            $
            &\quad
            $\alpha'=\alpha$
            \\ \hline
   P$_{3}$  &
            $
            \begin{cases}
            xy - yx + y^{2} - zx,\\
            xz + yz - zx,\\
            zy - yz - z^{2}
            \end{cases}
            $
            &\ \hfill --------------------- \hfill \rule{0pt}{10pt}
            \\ 
            \hline
            \end{tabular}}

\noindent{\renewcommand\arraystretch{1.5} 
\begin{tabular}{|p{0.7cm}|p{4.3cm}|p{6.3cm}|} \hline
S$_{1}$ &
$
\begin{cases}
yz-\alpha zy,\\
zx-\beta xz,\\
xy-\gamma yx
\end{cases}\hspace{-10pt}
$
($\alpha\beta\gamma \neq 0,1$)
&\quad 
\begin{minipage}{6cm}
\rule{0pt}{12pt}$
(\alpha',\beta',\gamma')
$\\
$=
\begin{cases}
(\alpha,\beta,\gamma),(\beta,\gamma,\alpha),\\
(\gamma,\alpha,\beta),(\alpha^{-1},\gamma^{-1},\beta^{-1}),\\
(\beta^{-1},\alpha^{-1},\gamma^{-1}),(\gamma ^{-1} , \beta ^{-1} ,\alpha ^{-1} )
\end{cases}$
\end{minipage}
\\ \hline
S$_{2}$ &
$
\begin{cases}
zx-\alpha yz,\\
xz-\beta zy,\\
x^{2}+\alpha \beta y^{2}
\end{cases}$\hspace{-10pt}
($\alpha\beta \neq 0$)
&\quad
$(\alpha':\beta')=(\alpha:\beta)\text{ in }\mathbb{P}^1$
\\ \hline
S$_{3}$ &
$
\begin{cases}
yx-\alpha z^{2},\\
zy - \beta x^{2},\\
xz-\gamma y^{2}
\end{cases}
$\hspace{-10pt}
($\alpha\beta\gamma\neq 0,1$)
&\quad
$\alpha'\beta'\gamma'=\alpha\beta\gamma$
\\ \hline\hline
S$'_{1}$  &
$
\begin{cases}
xy -\beta yx,\\
x^{2}+yz-\alpha zy,\\
zx-\beta xz \ \ \text{($\alpha\beta^{2} \neq 0,1$)}
\end{cases}
$
&\quad
$(\alpha',\beta')=(\alpha,\beta),(\alpha^{-1},\beta^{-1})$
\\ \hline
S$'_{2}$&
$
\begin{cases}
xy -zx,\\
yx - xz,\\
x^{2} + y^{2} +z^{2}
\end{cases}
$ 
&\ \hfill --------------------- \hfill \rule{0pt}{10pt}
\\ \hline
\end{tabular}}

\noindent{\renewcommand\arraystretch{1.5} 
\begin{tabular}{|p{0.7cm}|p{4.3cm}|p{6.3cm}|} \hline
	T$_{1}$ &
	\begin{minipage}{4cm}
		$
		\begin{cases}
		xy-yx,\\ 
		xz-zx-\beta x^{2} \\ \hspace{50pt} +(\beta+\gamma)yx, \\
		yz-zy-\alpha y^{2} \\ \hspace{50pt} +(\alpha+\gamma)xy\\
		\end{cases}
		$ 
		\\
		\hfill($\alpha+\beta+\gamma \neq 0$)
	\end{minipage}
	&\quad
\begin{minipage}{6cm}
$(\alpha':\beta':\gamma')$
\\
$=
	\begin{cases}
	(\alpha:\beta:\gamma),(\alpha:\gamma:\beta),\\
	(\beta:\alpha:\gamma),
	(\beta:\gamma:\alpha),\\
	(\gamma:\alpha:\beta),(\gamma:\beta:\alpha)
	\end{cases}
        \text{ in } \mathbb{P}^{2}
	$
\end{minipage}
	\\ \hline
	T$_{2}$  &
	\begin{minipage}{4cm}
		$
		\begin{cases}
		x^{2} - y^{2},\\
		xz-zy-\beta xy \\ \hspace{50pt} +(\beta+\gamma)y^{2},\\
		yz-zx-\alpha yx \\ \hspace{50pt} +(\alpha+\gamma)x^{2}
		\end{cases}
		$\\
		\hfill ($\alpha + \beta + \gamma \neq 0$)
	\end{minipage}
	&\quad
	$
	(\alpha'+\beta':\gamma')=(\alpha+\beta:\gamma)
	\text{ in }\mathbb{P}^1
	$
	\\ \hline
	T$_{3}$ &
	$
	\begin{cases}
	x^{2}-xy+y^{2},\\
	xz+zy,\\
	yx-yz+zx-zy
	\end{cases}
	$ 
	&\ \hfill --------------------- \hfill \rule{0pt}{10pt}
	\\ \hline
	\end{tabular}}

\noindent{\renewcommand\arraystretch{1.5} 
\begin{tabular}{|p{0.7cm}|p{4.3cm}|p{6.3cm}|} \hline
	T$'$   &
	\begin{minipage}{4cm}
	$
	\begin{cases}
	\alpha x^{2}+\beta(\alpha+\beta)xy-xz\\
	\hfill +zx-(\alpha+\beta)zy,\\
	xy-yx-\beta y^{2},\\ 
	2\beta xy-\beta^{2}y^{2}+yz-zy
	\end{cases}
	$ \\ \hfill ($\alpha+2\beta \neq 0$)
	\end{minipage}
	&\quad
	$(\alpha':\beta')=(\alpha:\beta)\text{ in }\mathbb{P}^{1}$
	\\ \hline
	\end{tabular}}

\noindent{\renewcommand\arraystretch{1.5} 
\begin{tabular}{|p{0.7cm}|p{4.3cm}|p{6.3cm}|} \hline
	CC &
	$
	\begin{cases}
	-3x^{2}-2xy+xz-zx \\ \hfill +2zy, \\
	-xy+yx+y^{2}, \\
	3x^{2}+y^{2}+yz-zy
	\end{cases}
	$
	&
	\ \hfill --------------------- \hfill \rule{0pt}{10pt}
	\\ \hline
	\end{tabular}}

\noindent{\renewcommand\arraystretch{1.5} 
\begin{tabular}{|p{0.7cm}|p{4.3cm}|p{6.3cm}|} \hline
    NC$_{1}$  &
              \begin{minipage}{4cm}
              $
              \begin{cases}
              xy-\alpha yx,\\
              \dfrac{\alpha^{3}-1}{\alpha}x^{2}+\alpha zy-yz, \\
              \dfrac{\alpha^{3}-1}{\alpha}y^{2}+\alpha xz-zx 
              \end{cases}
              $\\
              \hfill($\alpha(\alpha^{3}-1)\neq 0$)
              \end{minipage}
              &\quad
              $\alpha'=\alpha^{\pm 1}$ \hspace{4.5cm}
              \\ \hline
    NC$_{2}$  &
              $
              \begin{cases}
              xz-2yx+zy, \\
              zx-2xy+yz, \\
              y^{2}+x^{2}
              \end{cases}
              $
              &
              \ \hfill --------------------- \hfill \rule{0pt}{10pt}
              \\ \hline
              \end{tabular}}

\noindent{\renewcommand\arraystretch{1.5} 
\begin{tabular}{|p{0.7cm}|p{4.3cm}|p{6.3cm}|} \hline
     WL$_{1}$  &
\begin{minipage}{4cm}
$
               \begin{cases}
               \alpha xy-yx, \\
               \alpha xz-\gamma yx-zx, \\
               zy-yz+(1+\gamma)y^{2}
               \end{cases}
               $ 
              \\ \hfill ($\alpha \neq 0,1$)
\end{minipage}
               &\quad
               $(\alpha',\gamma')=(\alpha,\gamma)$
               \\ \hline
     WL$_{2}$  &
               $
               \begin{cases}
                xy-yx, \\
                xz-\gamma yx-zx, \\
                zy-yz+(1+\gamma)y^{2}
               \end{cases}
               $
               &\quad 
               $\gamma'=\gamma$
               \\ \hline
     WL$_{3}$  &
               $
               \begin{cases}
                xy-yx, \\
                xz-x^{2}-\gamma yx-zx, \\
                xy+zy-yz \\ \hspace{50pt} +(1+\gamma)y^{2}
               \end{cases}
               $
               &\quad 
               $\gamma'=\gamma$
               \\ \hline
\end{tabular}}

\noindent{\renewcommand\arraystretch{1.5} 
\begin{tabular}{|p{0.7cm}|p{4.3cm}|p{6.3cm}|} \hline
TL$_{1}$  &
\begin{minipage}{4cm}
$
\begin{cases}
xy-\alpha yx, \\
xz-\alpha^{-1}zx, \\
\alpha^{-1}zy-\alpha yz+x^{2}
\end{cases}
$
\\
\hfill ($\alpha\neq 0$)
\end{minipage}
&\quad 
$\alpha'=\alpha^{\pm 1}
$
\\ \hline
TL$_{2}$  &
$
\begin{cases}
xy-yx-\beta x^{2}, \\
xz-zx-yx, \\
zy- yz-\beta xz +x^{2}+y^{2}
\end{cases}
$
&\quad
$\beta'=\pm \beta$
\\ \hline
TL$_{3}$  &
$
\begin{cases}
xy+yx,\\
xz+zx-yx,\\
zy-yz-x^{2}-y^{2}
\end{cases}
$
&
\ \hfill --------------------- \hfill \rule{0pt}{10pt}
\\ 
\hline 
TL$_{4}$  &
$
\begin{cases}
xy+yx,\\
xz-zx-x^{2},\\
zy-yz+xy+x^{2}
\end{cases}
$
&\ \hfill --------------------- \hfill \rule{0pt}{10pt}
\\ 
\hline 
	\end{tabular}}
\vspace*{10pt}

The following theorem lists all possible defining relations of algebras 
in each type up to graded Morita equivalence except for Type EC.
\begin{thm} [\cite{E,K,KMM,Ma,O}]
\label{thm_main2}
Let $A=\A$ be a $3$-dimensional quadratic AS-regular algebra. 
For each type except for Type EC, the following table describes 
\begin{description}
\item[{\rm (I)}] the defining relations of $A$, and
\item[{\rm (III)}] the conditions to be graded Morita equivalent 
      in terms of their defining 
      relations. 
      {\rm (}see Example \ref{exm_Mo1}.  {\rm )}
\end{description}
In the following table,
if X $\neq$ Y, then Type X algebra 
is not graded Morita equivalent to any Type Y algebra. 
\end{thm}
\noindent{\renewcommand\arraystretch{1.5} 
\begin{tabular}{|p{0.7cm}|p{4.3cm}|p{6.3cm}|} \hline
 {\footnotesize Type}  &\quad (I) defining relations
       &\quad (II) condition to be  \\ 
 &\quad ($\alpha,\beta,\gamma \in k$) 
       &\quad graded algebra isomorphic \\ 
            \hline
\end{tabular}}

\noindent{\renewcommand\arraystretch{1.5} 
\begin{tabular}{|p{0.7cm}|p{4.3cm}|p{6.3cm}|} \hline
       P   &
            $
            \begin{cases}
            xy-yx ,\\
            yz-zy,\\
            zx-xz
            \end{cases}
            $
            &
\ \hfill --------------------- \hfill \rule{0pt}{10pt}
            \\ \hline
S &
$
\begin{cases}
yz-\alpha zy,\\
zx-\beta xz, \ \ \text{($\alpha\beta\gamma \neq 0,1$)} \\
xy-\gamma yx
\end{cases}
$
&\quad 
$\alpha'\beta'\gamma'=(\alpha\beta\gamma)^{\pm 1}$
\\ \hline
S$'$  &
$
\begin{cases}
xy -\beta yx,\\
x^{2}+yz-\alpha zy,\\
zx-\beta xz \ \ \ \text{($\alpha\beta^{2} \neq 0,1$)}
\end{cases}
$
&\quad
$\alpha'{\beta'}^{2}=(\alpha\beta^{2})^{\pm 1}$
\\ \hline
T &
\begin{minipage}{150pt}
	$
	\begin{cases}
	xy-yx,\\ 
	xz-zx-x^{2}, \\
	yz-zy-y^{2}\\
	\end{cases}
	$ 
\end{minipage}
&
\ \hfill --------------------- \hfill \rule{0pt}{10pt}
\\ \hline
    T$'$   &
               $
               \begin{cases}
               x^{2}-xz+zx-zy,\\
               xy-yx,\\ 
               yz-zy
               \end{cases}
               $ 
               &
\ \hfill --------------------- \hfill \rule{0pt}{10pt}
               \\ \hline

     CC &
             $
             \begin{cases}
              -3x^{2}-2xy+xz-zx \\ \hfill +2zy, \\
              -xy+yx+y^{2}, \\
              3x^{2}+y^{2}+yz-zy
              \end{cases}
             $
             &
\ \hfill --------------------- \hfill \rule{0pt}{10pt}
             \\ \hline
	NC  &
	\begin{minipage}{4cm}
		$
		\begin{cases}
		xy-\alpha yx,\\
		\dfrac{\alpha^{3}-1}{\alpha}x^{2}+\alpha zy-yz, \\
		\dfrac{\alpha^{3}-1}{\alpha}y^{2}+\alpha xz-zx 
		\end{cases}
		$\\
		\hfill ($\alpha(\alpha^{3}-1)\neq 0$)
	\end{minipage}
	&\quad 
	${\alpha'}^{3}=\alpha^{\pm 3}$
	\\ \hline
	WL  &
	$ 
	\begin{cases}
	xy+yx, \\
	xz+zx, \\
	zy-yz+y^{2}
	\end{cases}
	$ 
	&
\ \hfill --------------------- \hfill \rule{0pt}{10pt}
	\\ \hline
TL  &
$
\begin{cases}
xy-yx, \\
xz-zx, \\
zy-yz+x^{2}
\end{cases}
$
&
\ \hfill --------------------- \hfill \rule{0pt}{10pt}
\\ \hline
\end{tabular}}
\begin{rem}\label{Z-algs}
	Since ${\rm GrMod}\,A \cong {\rm GrMod}\,A'$ if and only if
	$\overline{A} \cong \overline{A'}$ as $\mathbb{Z}$-algebras
	where $\overline{A}:=\bigoplus_{i,j \in \mathbb{Z}} A_{j-i}$
	by \cite{Sie},
	the above table agrees with \cite[Table 1]{BP}. 
\end{rem}
If $E$ is reduced, then 
Theorem \ref{thm_main} and Theorem \ref{thm_main2} 
are proved by the following five steps (see Example \ref{exm_Mo1}): 
\begin{description}
\item[{\rm Step 0}] Fix a defining relation of $E$.
\item[{\rm Step 1}] Find all automorphisms $\sigma$ of $E$. 
\item[{\rm Step 2}] Find the defining relations of 
      $\mathcal{A}(E,\sigma)$ for each $\sigma\in {\rm Aut}_{k}E$
      by using (G2) condition in Definition \ref{geometric}. 
\item[{\rm Step 3}] Classify them up to isomorphism of graded algebras 
      in terms of their defining relations 
     by using Theorem \ref{gradediso}. 
\item[{\rm Step 4}] Classify them up to graded Morita equivalence 
      in terms of their defining relations 
      by using Theorem \ref{gradedequi}. 
\end{description}

For Type P$_{i}$ ($i=1,2,3$), Type S$_{i}$ ($i=1,2,3$), 
Type S$'_{i}$ ($i=1,2$), Type T$_{i}$ ($i=1,2,3$) and Type T$'$, 
the above five steps were completed in \cite{Ma} and \cite{KMM}. 
For Type CC and Type NC$_{i}$ ($i=1,2$), 
Step 1 was completed in \cite{O}, and 
Step 2, Step 3 and Step 4 were completed in \cite{E}. 

We briefly explain the method in \cite{O}. 
Let $E$ be an irreducible variety and $\pi\colon\Tilde{E}\to E$ 
a normalization of $E$. 
Then, for any $\sigma\in{\rm Aut}_{k}E$, 
there exists a unique $\Tilde{\sigma}\in{\rm Aut}_{k}\,\Tilde{E}$ 
such that $\sigma\circ\pi=\pi\circ \Tilde{\sigma}$,
i.e., the following diagram commutes{\rm :}
$$
\begin{CD}
\Tilde{E} @>{\pi}>> E\\
@V{\Tilde{\sigma}}VV @VV{\sigma}V \\
\Tilde{E}@>{\pi} >> E
\end{CD}
$$
In fact, the assignment $\sigma \longmapsto \Tilde{\sigma}$ is 
an injective group homomorphism 
from ${\rm Aut}_{k}E$ to ${\rm Aut}_{k}\Tilde{E}$. 

For example, let $A=\mathcal{A}(E,\sigma)$ be a Type NC algebra.\\
Step 0: Since $E$ is a nodal cubic curve,
we may assume that $E=\mathcal{V}(x^{3}+y^{3}+xyz)$.\\
Step 1: A normalization $\pi\colon\mathbb{P}^{1}=\Tilde{E}\longrightarrow E$ is given by
	$$
	\pi(a\colon b)=(a^{2}b\colon ab^{2}\colon -a^{3}-b^{3}). 
	$$
	Since $\sigma$ fixes the singular point $(0:0:1) \in E$ 
        and $\pi^{-1}((0:0:1))=\{ (1:0),(0:1) \}
	\subset \mathbb{P}^{1}$, 
        either $\tilde{\sigma}$ fixes both $(1:0)$ and $(0:1)$
	so that
	$
	\tilde{\sigma}
	=\left(
	\begin{array}{cc}
	1 & 0 \\
	0 & \alpha
	\end{array}
	\right)
	$ 
	for $0 \neq \al \in k$,
	or $\tilde{\sigma}$ switches $(1:0)$ and $(0:1)$ so that
	$
	\tilde{\sigma}
	=\left(
	\begin{array}{cc}
	0 & 1 \\
	\beta & 0
	\end{array}
	\right)
	$
	for $0 \neq \beta \in k$.
	In each case, the corresponding $\sigma$ is given as
	$$
	\sigma_{1}(x\colon y\colon z)
	=(\alpha xy \colon \alpha^{2}y^{2} \colon (\alpha^{3}-1)x^{2}+\alpha^{3}yz) 
        \quad (\alpha^{3}\ne 0,1)
	$$
	or
	$$
	\sigma_{2}(x\colon y\colon z)
	=(\beta y^{2} \colon \beta^{2}xy \colon (1-\beta^{3})x^{2}+yz) 
        \quad (\beta^{3}\ne 0,1). 
	$$
\begin{rem}
\label{rem_typeNC} 
We call the above $\mathcal{A}(E,\sigma_{i})$ 
{\it Type NC$_{i}$ algebras} ($i=1,2$). 
Type NC$_{1}$ algebras are isomorphic to algebras 
given in \cite[Theorem 2.2]{NVZ}, 
however, Type NC$_{2}$ algebras are not isomorphic to any algebra 
in \cite[Theorem 2.2]{NVZ}. 
In fact, the above $\sigma_{1}$ was in \cite{NVZ}, 
but $\sigma_{2}$ was overlooked in \cite{NVZ}. 
\end{rem}
To prove Theorem \ref{thm_main} and Theorem \ref{thm_main2} 
when $E$ is a non-reduced cubic in $\mathbb{P}^{2}$,  
we use the following key lemma. 
\begin{lem}[{\cite[Theorem 8.16 (iii)]{ATV2}}]
\label{twistWLTL}
\begin{enumerate}[{\rm (1)}]
\item 
If $A$ is a $3$-dimensional quadratic AS-regular algebra of Type $WL$, 
then there exists $\varphi \in {\rm GrAut}_{k}A$ such that 
\begin{align*} 
&A^{\varphi} \cong B_{1}:=k\langle x,y,z\rangle/(xy-yx,xz-zx,zy-yz+xz),
\text{ or }\\
&A^{\varphi} \cong B_{2}:=k\langle x,y,z\rangle/(xy-yx,xz-zx,zy-yz+y^{2}). 
\end{align*}
\item 
If $A$ is a $3$-dimensional quadratic AS-regular algebra of Type TL, 
then there exists $\varphi \in {\rm GrAut}_{k}A$
such that 
$$
A^{\varphi} \cong B_{3}:=k\langle x,y,z\rangle/(xy-yx,xz-zx,zy-yz+x^{2}). 
$$
\end{enumerate}
\end{lem} 
Since $B=A^{\varphi}$ if and only if 
$A=B^{\varphi^{-1}}$ by \cite[Proposition 2.5 (2)]{Z}, 
for Type WL algebras and Type TL algebras, 
Theorem \ref{thm_main} and Theorem \ref{thm_main2} are proved 
by the following four steps:  
\begin{description}
\item[{\rm Step 1}] Find all graded algebra automorphisms $\varphi^{-1}$ of $B_{i}$ ($i=1,2,3$) 
                    in Lemma \ref{twistWLTL}. 
\item[{\rm Step 2}] Find the defining relations of $B_{i}^{\varphi^{-1}}$
                    by using Definition \ref{twa}.
\item[{\rm Step 3}] Classify them up to isomorphism of graded algebras 
                    in terms of their defining relations 
                    by using Lemma \ref{lem1}. 
\item[{\rm Step 4}] Classify them up to graded Morita equivalence 
                    in terms of their defining relations 
                    by using Theorem \ref{twist}. 
\end{description}
Step 1 and Step 2 were completed in \cite{K} and, 
Step 3 and Step 4 were completed in \cite{E}. 
\section{Defining relations for Type EC algebras}
\label{sec_EC}
Throughout this section, 
let $E$ be an elliptic curve in $\mathbb{P}^{2}$. 
Our aim in this section is to find ${\rm Aut}_{k}\,E$ and 
to compute the defining relations of $\mathcal{A}(E,\sigma)$ 
where $\sigma \in {\rm Aut}_{k}\,E$. 

It is well-known that the $j$-invariant $j(E)$ classifies elliptic curves 
up to projective equivalence. 
\begin{thm}[{\cite[Theorem IV 4.1 (b)]{H}}]\label{pe}
Let $E$ and $E'$ be two elliptic curves in $\mathbb{P}^{2}$. 
Then $E$ and $E'$ are projectively equivalent if and only if $j(E)=j(E')$. 
\end{thm}
Let $X$ be a scheme and $Y$ a subscheme of $X$. 
We define 
$$
{\rm Aut}_{k}\,(X,Y)
:=\{\phi\in {\rm Aut}_{k}\,X\mid \phi|_{Y}\in {\rm Aut}_{k}\,Y \}. 
$$
We view an element of ${\rm Aut}_{k}\,(X,Y)$ in two ways, 
that is, 
as an automorphism of $X$ which restricts to 
an automorphism of $Y$ 
and as an automorphism of $Y$ which extends to 
an automorphism of $X$.
In particular, if $Y=\{p\}$, 
then we write ${\rm Aut}_{k}\,(X,Y)={\rm Aut}_{k}\,(X,p)$. 
\begin{thm}[{\cite[Corollary IV 4.7]{H}}]\label{Ha}
	Let $E$ be an elliptic curve in $\mathbb{P}^{2}$. 
        For every $p \in E$, 
	$$|{\rm Aut}_{k}(E,p)|=
	\begin{cases}
	2 &\quad\quad  \text{if } j(E)\neq 0, 12^{3},\\
	6 &\quad\quad  \text{if } j(E)=0,\\
	4 &\quad\quad  \text{if } j(E)= 12^{3}.
	\end{cases}$$
\end{thm}
For each point $o \in E$, we can define an addition on $E$ so that
$E$ is an abelian group with the identity element $o$ and,
for $p \in E$, 
the map $\sigma_{p}$ defined by $\sigma_{p}(q):=p+q$
is a scheme automorphism of $E$,
called the \textit{translation} by a point $p$.
We write $(E,o)$ 
when we view $E$ as an abelian group with the identity element $o \in E$.

In this paper, we use a {\it Hesse form}
$E_{\la}:=\mathcal{V}(x^{3}+y^{3}+z^{3}-3\lambda xyz)$
where $\lambda \in k$.
It is known that $E_{\la}$ is an elliptic curve in $\mathbb{P}^{2}$ if and only if $\la^{3} \neq 1$.
The {\it $j$-invariant} of $E_{\la}$ is given by the formula
$$
j(E_{\la})=\dfrac{27\lambda^{3}(\lambda^{3}+8)^{3}}{(\lambda^{3}-1)^{3}} 
$$
(\cite[Proposition 2.16]{F}).

Every elliptic curve in $\mathbb{P}^{2}$ 
is projectively equivalent to $E_{\la}$ for some $\la$
with $\la^{3} \neq 1$ (\cite[Corollary 2.18]{F}). 
\begin{thm}[{\cite[Theorem 2.11]{F}}]
	Let $E_{\la}$ be an elliptic curve of a Hesse form in $\mathbb{P}^{2}$ 
        and $o_{\la}:=(1:-1:0) \in E_{\la}$.
	The group structure on $(E_{\la},o_{\la})$ is given as follows {\rm :}
	for $p=(a:b:c)$ and $q=(\al:\be:\ga) \in E_{\la}$,
	$$
	p+q:=
	\begin{cases}
		(ac\be^{2}-b^{2}\al\ga:bc\al^{2}-a^{2}\be\ga:ab\ga^{2}-c^{2}\al\be) \quad &{\rm if}\,\,\,p \neq q,\\
	\\
		(a^{3}b-bc^{3}:ac^{3}-ab^{3}:b^{3}c-a^{3}c) \quad &{\rm if}\,\,\,p=q.
    \end{cases}
    $$
\end{thm}
Throughout this paper, 
we fix the above group structure on $E_{\la}$ 
with the identity $o_{\la}:=(1:-1:0) \in E_{\la}$.
\subsection{Automorphism groups}
\begin{lem}[{\cite[Lemma IV 4.9]{H}}]\label{H 4.9}
	Let $(E,o)$ and $(E',o')$ be two elliptic curves in $\mathbb{P}^{2}$. 
	If $\varphi:E \rightarrow E'$ is a morphism of schemes sending $o$ to $o'$, 
	then $\varphi$ is also a group homomorphism.
\end{lem}
We set the following notations:
\begin{enumerate}[{\rm (i)}]
	\item $T:=\{ \sigma_{p}\in{\rm Aut}_{k}\,E\,|\,p \in E \}$ and
	$T_{\la}:=\{ \sigma_{p} \in {\rm Aut}_{k}\,E_{\la}\,|\,p \in E_{\la} \}$.
	\item $G:={\rm Aut}_{k}(E,o)$ and $G_{\la}:={\rm Aut}_{k}(E_{\la},o_{\la})$.
\end{enumerate}
For $\sigma_{p} \in T$ and $\tau \in G$,
it is easy to see that $\tau\sigma_{p}\tau^{-1}=\sigma_{\tau(p)} \in T$.
\begin{prop}[{cf. \cite[Section 6]{BP}}]\label{auto}
Suppose that $(E,o)$ is an elliptic curve in $\mathbb{P}^{2}$.
If $\Phi:G \rightarrow {\rm Aut}\,T$ is the group homomorphism
defined by $\Phi_{\tau}(\sigma_{p})=\sigma_{\tau(p)}$ for $\tau \in G$ and $\sigma_{p} \in T$,
then ${\rm Aut}_{k}\,E \cong T\rtimes_{\Phi} G$.
\end{prop}
\begin{thm}\label{generat}
	Let $E_{\la}$ be an elliptic curve in $\mathbb{P}^{2}$.
	A generator $\tau_{\la}$ of $G_{\la}$ is given by
	$$
	\begin{cases}
	\tau_{\la}(a:b:c):=(b:a:c) 
	\hfill \text{if } j(E_{\la}) \neq 0,12^3, \\
	\\
	\tau_{\la}(a:b:c):=(b:a:c\ep)
	\hfill \text{if } \la=0 \,\,(\text{so that }j(E_{\la})=0), \\
	\\
	\tau_{\la}(a:b:c):=(a\ep^2+b\ep+c:a\ep+b\ep^2+c:a+b+c)
	\\ \hspace{5cm} \text{if } \la=1+\sqrt{3} 
        \,\,(\text{so that }j(E_{\la})=12^{3}), 
	\end{cases}
	$$
	where $\ep$ is a primitive $3$rd root of unity.
	In particular, $G_{\la}$ is the subgroup of 
        ${\rm Aut}_{k}(\mathbb{P}^{2},E_{\la})$.
\end{thm}
\begin{proof}
\begin{enumerate}[(i)]
\item If $j(E_{\la})\neq 0,12^{3}$, then $|G_{\la}|=2$ by Theorem \ref{Ha}.
Let
$
\tau_{\la}=
\begin{pmatrix}
0&1&0\\
1&0&0\\
0&0&1
\end{pmatrix} \in {\rm PGL}_{3}(k) \cong {\rm Aut}_{k}\,\mathbb{P}^{2}$. 
If $p=(a:b:c) \in E_{\la}$, then $\tau_{\la}(p)=(b:a:c) \in E_{\la}$, so
$\tau_{\la} \in {\rm Aut}_{k}(\mathbb{P}^{2},E_{\la})$.
Since $\tau_{\la}(o_{\la})=o_{\la}$, we have $\tau_{\la} \in G_{\la}$.
By calculations,
$|\tau_{\la}|=2$, so $G_{\la}= \langle \tau_{\la} \rangle $.
\item If $\la=0$ so that $E_{\la}=\mathcal{V}(x^3+y^3+z^3)$,
then $j(E_{\la})=0$, so $|G_{\la}|=6$ by Theorem \ref{Ha}.
Let
$
\tau_{\la}=
\begin{pmatrix}
0&1&0\\
1&0&0\\
0&0&\ep
\end{pmatrix} \in {\rm PGL}_{3}(k) \cong {\rm Aut}_{k}\,\mathbb{P}^{2}
$, 
where $\ep$ is a primitive $3$rd root of unity.
If $p=(a:b:c) \in E_{\la}$, then $\tau_{\la}(p)=(b:a:c\ep) \in E_{\la}$,
so $\tau_{\la} \in {\rm Aut}_{k}(\mathbb{P}^{2},E_{\la})$.
Since $\tau_{\la}(o_{\la})=o_{\la}$, we have $\tau_{\la} \in G_{\la}$.
By calculations,
$|\tau_{\la}|=6$, so $G_{\la}= \langle \tau_{\la} \rangle$.
\item If $\la=1+\sqrt{3}$ so that $E_{\la}=\mathcal{V}(x^3+y^3+z^3-3(1+\sqrt{3})xyz)$, 
then $j(E_{\la})=12^{3}$, so $|G_{\la}|=4$ by Theorem \ref{Ha}.
Let
$
\tau_{\la}=
\begin{pmatrix}
\ep^{2}&\ep&1\\
\ep&\ep^{2}&1\\
1&1&1
\end{pmatrix} \in {\rm PGL}_{3}(k) \cong {\rm Aut}_{k}\,\mathbb{P}^{2}
$.
If $p=(a:b:c) \in E_{\la}$,
then $\tau_{\la}(p)=(a\ep^{2}+b\ep+c:a\ep+b\ep^{2}+c:a+b+c)$.
Since
\begin{align*}
	(a\ep^{2}+b\ep+c)^{3}&+(a\ep+b\ep^{2}+c)^{3}+(a+b+c)^{3}\\
	-3&(1+\sqrt{3})(a\ep^{2}+b\ep+c)(a\ep+b\ep^{2}+c)(a+b+c)\\
	&=3(a^{3}+b^{3}+c^{3})+18abc-3(1+\sqrt{3})(a^{3}+b^{3}+c^{3}-3abc)\\
	&=-3\sqrt{3}(a^{3}+b^{3}+c^{3})+9\sqrt{3}(1+\sqrt{3})abc\\
	&=-3\sqrt{3}(a^{3}+b^{3}+c^{3}-3(1+\sqrt{3})abc)\\
	&=0,
\end{align*}
we have $\tau_{\la}(p) \in E_{\la}$, so $\tau_{\la} \in {\rm Aut}_{k}(\mathbb{P}^{2},E_{\la})$.
Since $\tau_{\la}(o_{\la})=o_{\la}$, we have $\tau_{\la} \in G_{\la}$.
By calculations,
$|\tau_{\la}|=4$, so $G_{\la}=\langle \tau_{\la} \rangle$.
\end{enumerate}
\end{proof}

We fix the above generator $\tau_{\la}$ of $G_{\la}$
for the rest of the paper.
\subsection{Defining Relations} 
\begin{lem}\label{isomorphic}
	Every $3$-dimensional quadratic ${\rm AS}$-regular algebra 
        $A=\mathcal{A}(E,\sigma)$ of {\rm Type EC}
	is isomorphic to $\mathcal{A}(E_{\la},\sigma_{p}\tau_{\la}^{i})$
	where $\la \in k$ with $\la^{3} \neq 1$, $p \in E_{\la}$ 
        and $i \in \mathbb{Z}$ . 
\end{lem}
\begin{proof}
    By Theorem \ref{pe}, there exists $\la \in k$ such that 
    $E$ and $E_{\la}$ are projectively equivalent.
    If we set $\sigma':=\phi\sigma\phi^{-1} \in {\rm Aut}_{k}\,E_{\la}$
    where $\phi:E \rightarrow E_{\la}$ is a projective equivalence,
    then the diagram
    \[ \xymatrix{
    	E \ar[r]^{\phi} \ar[d]_{\sigma} & E_{\la} \ar[d]^{\sigma'} \\
    	E \ar[r]^{\phi} & E_{\la}
    }\]
    commutes, so $\mathcal{A}(E,\sigma) \cong \mathcal{A}(E_{\la},\sigma')$ 
    by \cite[Lemma 2.6 (1)]{MU}. 
    By Proposition \ref{auto} and Theorem \ref{generat}, 
    there exist $p \in E_{\la}$ and $i \in \mathbb{Z}$ such that
    $\sigma'=\sigma_{p}\tau_{\la}^{i}$ 
    where $\langle \tau_{\la} \rangle = G_{\la}={\rm Aut}_{k}(E_{\la},o_{\la})$, 
    so $A \cong \mathcal{A}(E_{\la},\sigma_{p}\tau_{\la}^{i})$.
\end{proof}
We can compute the defining relations 
of $3$-dimensional quadratic AS-regular algebras of Type EC
by using the defining relations of a $3$-dimensional Sklyanin algebra
$$
\mathcal{A}(E,\sigma_{p}) = k\langle x,y,z \rangle
/(ayz+bzy+cx^2,azx+bxz+cy^2,axy+byx+cz^2)
$$
where $p=(a:b:c) \in \mathbb{P}^{2}$. 
We say that a geometric algebra $A$ is {\it of Type EC} 
if the point scheme of $A$ is an elliptic curve. 
\begin{lem}
\label{lem_abc}
Let $E_{\la}$ be an elliptic curve in $\mathbb{P}^{2}$ where $\la^{3} \neq 1$,
$p=(a:b:c) \in E_{\la}$ and $i \in \mathbb{Z}$. Then
$\mathcal{A}(E_{\la},\sigma_{p}\tau_{\la}^{i})$ 
is a geometric algebra of Type EC 
if and only if $abc\neq 0$. 
\end{lem}
\begin{proof}
If $abc \neq 0$, then $((a^3+b^3+c^3)/3abc)^3=\la^3\neq1$,
that is, $(a^3+b^3+c^3)^{3} \neq (3abc)^{3}$, so
$\mathcal{A}(E_{\la},\sigma_{p})$ is a $3$-dimensional
quadratic AS-regular algebra of Type EC by \cite[Section 1]{ATV1}.
Since $\mathcal{A}(E_{\la},\sigma_{p})$ is a geometric algebra of Type EC,
$\mathcal{A}(E_{\la},\sigma_{p}\tau_{\la}^{i})$ is
also a geometric algebra of Type EC by Proposition \ref{ge} (3).
If $abc=0$, then the point scheme of $\mathcal{A}(E_{\la},\sigma_{p})$
is $\mathbb{P}^{2}$ by \cite[Section 1]{ATV1}, so
$\mathcal{A}(E_{\la},\sigma_{p}\tau_{\la}^{i})$ is not of Type EC
by Proposition \ref{ge} (2).
\end{proof}
\begin{thm}\label{df}
	Every $3$-dimensional quadratic ${\rm AS}$-regular algebra
	$\mathcal{A}(E,\sigma)$
	of Type ${\rm EC}$ is isomorphic to one of the following algebras 
        $\kang/(f_{1},f_{2},f_{3})${\rm :}
	\begin{enumerate}[{\rm (1)}]
		\item If $j(E) \neq 0,12^{3}$, then 
	\begin{align*}
		&\begin{cases}
		f_{1}=ayz+bzy+cx^2, \\
		f_{2}=azx+bxz+cy^2, \\
		f_{3}=axy+byx+cz^2.
		\end{cases}
	    &\begin{cases}
		 f_1=axz+bzy+cyx, \\
		 f_2=azx+byz+cxy, \\
		 f_3=ay^2+bx^2+cz^2.
		\end{cases}
	\end{align*}
	    where $(a:b:c) \in E_{\la}$ with $j(E_{\la})=j(E)$ 
            such that $abc \neq 0$.
	    
		\item If $j(E)=0$, then
		\begin{align*}
		&\begin{cases}
		f_1=ayz+bzy+cx^2, \\
		f_2=azx+bxz+cy^2, \\
		f_3=axy+byx+cz^2.
	    \end{cases}
	     &&\begin{cases}
	     f_1=axz+b\ep zy+cyx, \\
	     f_2=a\ep zx+byz+cxy, \\
	     f_3=ay^2+bx^2+c\ep z^2. 
	    \end{cases} \\
	    &\begin{cases}
		f_1=ayz+b\ep^{2}zy+cx^2, \\
		f_2=a\ep^{2}zx+bxz+cy^2, \\
		f_3=axy+byx+c\ep^{2}z^2.
		\end{cases}
		 &&\begin{cases}
		f_1=axz+bzy+cyx, \\
		f_2=azx+byz+cxy, \\
		f_3=ay^2+bx^2+cz^2. 
		\end{cases} \\
		&\begin{cases}
		f_1=ayz+b\ep zy+cx^2, \\
		f_2=a\ep zx+bxz+cy^2, \\
		f_3=axy+byx+c\ep z^2. 
		\end{cases}
		 &&\begin{cases}
		f_1=axz+b\ep^{2}zy+cyx, \\
		f_2=a\ep^{2}zx+byz+cxy, \\
		f_3=ay^2+bx^2+c\ep^{2}z^2.
		\end{cases}
		\end{align*}
		where $(a:b:c) \in E_{0}$ such that $abc \neq 0$ 
                and $\ep$ is a primitive $3$rd root of unity.
		
		\item If $j(E)=12^{3}$, then
		\begin{align*}
		&\begin{cases}
		f_1=ayz+bzy+cx^2, \\
		f_2=azx+bxz+cy^2, \\
		f_3=axy+byx+cz^2.  
		\end{cases}
		\begin{cases}
		f_1=a(\ep x+\ep^{2}y+z)z+b(x+y+z)y
		\\ \hfill +c(\ep^{2}x+\ep y+z)x, \\
		f_2=a(x+y+z)x+b(\ep^{2}x+\ep y+z)z
		\\ \hfill +c(\ep x+\ep^{2}y+z)y, \\
		f_3=a(\ep^{2}x+\ep y+z)y+b(\ep x+\ep^{2}y+z)x
		\\ \hfill +c(x+y+z)z.
		\end{cases} \\
		&\begin{cases}
		f_1=axz+bzy+cyx, \\
		f_2=azx+byz+cxy, \\
		f_3=ay^2+bx^2+cz^2. 
		\end{cases}
		\begin{cases}
		f_1=a(\ep^{2}x+\ep y+z)z+b(x+y+z)y
		\\ \hfill +c(\ep x+\ep^{2}y+z)x, \\
		f_2=a(x+y+z)x+b(\ep x+\ep^{2}y+z)z
		\\ \hfill +c(\ep^{2}x+\ep y+z)y, \\
	    f_3=a(\ep x+\ep^{2}y+z)y+b(\ep^{2}x+\ep y+z)x
	    \\ \hfill +c(x+y+z)z.
		\end{cases} 
		\end{align*}
		where $(a:b:c) \in E_{1+\sqrt{3}}$ 
                such that $abc \neq 0$ and $\ep$ is a primitive $3$rd root of unity.
	\end{enumerate}
\end{thm}
\begin{proof}
	Let $A$ be a $3$-dimensional quadratic AS-regular algebra of Type EC.
	By Lemma \ref{isomorphic} and Proposition \ref{ge} (1), 
	there exist $\la \in k$ with $\la^{3} \neq 1$, 
        $p=(a:b:c) \in E_{\la}$ and $i \in \mathbb{Z}$
	such that 
        $
        A \cong \mathcal{A}(E_{\la},\sigma_{p}\tau_{\la}^{i})
        =\mathcal{A}(E_{\la},\sigma_{p}\overline{\phi_{\la}^{\ast}}^{i})
        =\mathcal{A}(E_{\la},\sigma_{p})^{{\phi}_{\la}^{i}}
        $
	where $\phi_{\la} \in {\rm GL}_{3}(k)$ is given by
	$$\phi_{\la}:=
	\begin{cases}
	\begin{pmatrix}
	0 & 1 & 0 \\
	1 & 0 & 0 \\
	0 & 0 & 1
	\end{pmatrix} \quad &\text{ if } j(E_{\la}) \neq 0,12^3,
	\\
	\begin{pmatrix}
	0 & 1 & 0 \\
	1 & 0 & 0 \\
	0 & 0 & \ep
	\end{pmatrix} \quad &\text{ if } \la =0,
	\\
	\begin{pmatrix}
	\ep^{2} & \ep & 1 \\
	\ep & \ep^{2} & 1 \\
	1 & 1 & 1
	\end{pmatrix} \quad &\text{ if } \la=1+\sqrt{3}.
	\end{cases}$$ \\
	By Lemma \ref{lem_abc}, $abc \neq 0$ and,
	by the definition of a twisted algebra (see Definition \ref{twa}),
	the defining relations of $\mathcal{A}(E_{\la},\sigma_{p})^{\phi_{\la}^{i}}$ 
    are given by
	$$
	a\phi_{\la}^{i}(y)z+b\phi_{\la}^{i}(z)y+c\phi_{\la}^{i}(x)x,
	$$$$
	a\phi_{\la}^{i}(z)x+b\phi_{\la}^{i}(x)z+c\phi_{\la}^{i}(y)y,
	$$$$
	a\phi_{\la}^{i}(x)y+b\phi_{\la}^{i}(y)x+c\phi_{\la}^{i}(z)z.
	$$
    Thus $A$ is isomorphic to one of the listed algebras in the statement.
\end{proof}

\begin{rem}
	Unfortunately, not every algebra listed in Theorem \ref{df} is AS-regular,
	so Theorem \ref{df} does not give a complete 
        list of $3$-dimensional AS-regular algebras of Type EC,
	but a complete list of geometric algebras of Type EC.
	In a subsequent paper \cite{IM},
	we give a geometric characterization of AS-regularity of algebras
	listed in Theorem \ref{df}.
\end{rem}
\subsection{Classification up to graded algebra isomorphism}
By \cite[Corollary 2.18]{F}, for every $E$, 
there exists $\la \in k$ with $\la^{3} \neq 1$ such that $E_{\la}$ and $E$ 
are projectively equivalent.
If $\psi:E_{\la} \rightarrow E$ is a projective equivalence, then
$\Psi:{\rm Aut}_{k}\,E_{\la} \rightarrow {\rm Aut}_{k}\,E$
defined by $\Psi(\sigma):=\psi\sigma\psi^{-1}$ is a group isomorphism.
If $o:=\psi(o_{\la})$, then $\psi:(E_{\la},o_{\la}) \rightarrow (E,o)$
is a group isomorphism by Lemma \ref{H 4.9}, and
$\Psi(\sigma_{p})=\psi\sigma_{p}\psi^{-1}=\sigma_{\psi(p)} \in T$
for $\sigma_{p} \in T_{\la}$.
For the rest paper, we fix
\begin{enumerate}[{\rm (a)}]
	\item a projective equivalence $\psi:E_{\la} \rightarrow E$,
	
	\item the group isomorphism
	$\Psi:{\rm Aut}_{k}\,E_{\la} \rightarrow {\rm Aut}_{k}\,E$
	defined by 
        $$
        \Psi(\sigma):=\psi\sigma\psi^{-1}
        , $$
	\item the identity element $o:=\psi(o_{\la})$ of $E$, and
	
	\item the generator $\tau:=\Psi(\tau_{\la})$ of $G={\rm Aut}_{k}(E,o)$.
\end{enumerate}
We set the following notations:
\begin{enumerate}[{\rm (i)}]
		\item $E[3]:=\{ p \in E\,|\,3p=o \}$ 
                and $E_{\la}[3]:=\{ p \in E_{\la}\,|\,3p=o_{\la} \}$.
		\item $T[3]:=\{ \sigma \in T\,|\,\sigma^{3}={\rm id}_{E} \}
		=\{ \sigma_{p} \in T\,|\,p \in E[3] \}$ and \\
		$T_{\la}[3]:=\{ \sigma \in T_{\la}\,|\,\sigma^{3}={\rm id}_{E_{\la}} \}
		=\{ \sigma_{p} \in T_{\la}\,|\,p \in E_{\la}[3] \}$.
		\item $d:=|G|$ and $d_{\la}:=|G_{\la}|$.
		\item $F_{i}:=\{ p-\tau^{i}(p) \in E\,|\,p \in E[3] \}$ 
                for $i \in \mathbb{Z}_{d}$ and \\
		$F_{\la,i}:=\{ p-\tau_{\la}^{i}(p) \in E_{\la}\,|\,p \in E_{\la}[3] \}$
		for $i \in \mathbb{Z}_{d_{\la}}$.
	\end{enumerate}
It is easy to check the following lemma.

\begin{lem}\label{lem2}
	The following hold. 
	\begin{enumerate}[{\rm (1)}]
		\item $E[3]=\psi(E_{\la}[3])$.
		\item $F_{i}=\psi(F_{\la,i})$.
		\item ${\rm Aut}_{k}(\mathbb{P}^{2},E)=\Psi({\rm Aut}_{k}(\mathbb{P}^{2},E_{\la}))$.
		\item $G=\Psi(G_{\la})$.
		\item $T=\Psi(T_{\la})$.
		\item $T[3]=\Psi(T_{\la}[3])$.
	\end{enumerate}
\end{lem}
\begin{thm}\label{PE}
	The following hold. 
	\begin{enumerate}[{\rm (1)}]
		\item ${\rm Aut}_{k}(\mathbb{P}^{2},E) \cap T=T[3]$.
		\item $G \leq {\rm Aut}_{k}(\mathbb{P}^{2},E)$.
		\item ${\rm Aut}_{k}(\mathbb{P}^{2},E) \cong T[3] \rtimes G$.
	\end{enumerate}
\end{thm}

\begin{proof}
	\begin{enumerate}[{\rm (1)}]
		\item See \cite[Lemma 5.3]{Mo1}.
		
		\item Since $G_{\la} \leq {\rm Aut}_{k}(\mathbb{P}^{2},E_{\la})$
		by Theorem \ref{generat},
		$
		G=\Psi(G_{\la}) \leq \Psi({\rm Aut}_{k}(\mathbb{P}^{2},E_{\la}))
		$
		$={\rm Aut}_{k}(\mathbb{P}^{2},E)
		$
		by Lemma \ref{lem2} (3) and (4).
		
		\item Since ${\rm Aut}_{k}\,E \cong T \rtimes G$ by Proposition \ref{auto}
		and $G \leq {\rm Aut}_{k}(\mathbb{P}^{2},E)$ by ($2$),
		for $\sigma_{p}\tau^{i} \in {\rm Aut}_{k}\,E$,
		$\sigma_{p}\tau^{i} \in {\rm Aut}_{k}(\mathbb{P}^{2},E)$ if and only if
		$\sigma_{p} \in {\rm Aut}_{k}(\mathbb{P}^{2},E)$
                if and only if
		$\sigma_{p} \in T[3]$ by ($1$), 
                so ${\rm Aut}_{k}(\mathbb{P}^{2},E) \cong T[3] \rtimes G$.
		 
	\end{enumerate}
\end{proof}
\begin{rem}
	Theorem \ref{PE} ($2$) depends of the special choice of the identity element $o \in E$.
	In fact, if we choose an arbitrary point $p \in E$, then it is hardly the case that
	${\rm Aut}_{k}(E,p) \leq {\rm Aut}_{k}(\mathbb{P}^{2},E)$.
\end{rem}
\begin{lem}
\label{lem_E[3]}
Let $E$ be an elliptic curve in $\mathbb{P}^{2}$, $p \in E$
and $i \in \mathbb{Z}$. Then
$\mathcal{A}(E,\sigma_{p}\tau^{i})$ is a geometric algebra of Type EC 
if and only if 
$p\in E \setminus E[3]$. 
\end{lem}
\begin{proof}
	For $q=(a:b:c) \in E_{\la}$,
	$q \in E_{\la} \setminus E_{\la}[3]$ if and only if $abc \neq 0$
	if and only if $\mathcal{A}(E_{\la},\sigma_{q}\tau_{\la}^{i})$
	is a geometric algebra of Type EC by Lemma \ref{lem_abc}, so
	\begin{align*}
	\mathcal{A}(E,\sigma_{p}\tau^{i}) 
        &\cong \mathcal{A}(E_{\la},\Psi^{-1}(\sigma_{p}\tau^{i}))\\
	&=\mathcal{A}(E_{\la},\Psi^{-1}(\sigma_{p})\Psi^{-1}(\tau)^{i})\\
	&=\mathcal{A}(E_{\la},\sigma_{\psi^{-1}(p)}\tau_{\la}^{i})
	\end{align*}
	is a geometric algebra of Type EC if and only if 
        $\psi^{-1}(p) \in E_{\la} \setminus E_{\la}[3]$
	if and only if $p \in E \setminus E[3]$.
\end{proof}
We use the following two formulas. 
\begin{lem}\label{formula}
For $\sigma_{p}\tau^{i}, \sigma_{q}\tau^{j}$ 
and $\sigma_{r}\tau^{l} \in {\rm Aut}_{k}\,E$,
\begin{equation}
	(\sigma_{q}\tau^{j})(\sigma_{r}\tau^{l})(\sigma_{p}\tau^{i})^{-1}
	=\sigma_{q+\tau^{j}(r)-\tau^{l+j-i}(p)}\tau^{l+j-i}, \label{formula1}
\end{equation}
and
\begin{equation}
	(\sigma_{q}\tau^{j})^{-1}(\sigma_{r}\tau^{l})(\sigma_{p}\tau^{i})
	=\sigma_{\tau^{-j}(-q+r+\tau^{l}(p))}\tau^{l+i-j}. \label{formula2}
\end{equation}
\end{lem}

\begin{proof}
	By calculations.
\end{proof}
By Proposition \ref{auto},
for $\sigma_{p}\tau^{i}, \sigma_{q}\tau^{j} \in {\rm Aut}_{k}\,E \cong T \rtimes G$,
$\sigma_{p}\tau^{i}=\sigma_{q}\tau^{j}$ 
if and only if $p=q$ in $E$ and $i=j$ in $\mathbb{Z}_{d}$, 
\begin{thm}\label{classify}
	Let $E$ be an elliptic curve in $\mathbb{P}^{2}$, $p,q \in E \setminus E[3]$ 
        and $i,j \in \mathbb{Z}_{d}$.
	Then $\mathcal{A}(E,\sigma_{p}\tau^{i}) \cong \mathcal{A}(E,\sigma_{q}\tau^{j})$
	if and only if $i=j$ and $q=\tau^{l}(p)+r$ 
        where $r \in F_{i}$ and $l \in \mathbb{Z}_{d}$.
\end{thm}
\begin{proof}
	Since $\mathcal{A}(E,\sigma_{p}\tau^{i})$ 
        and $\mathcal{A}(E,\sigma_{q}\tau^{j})$
	are geometric algebras of Type EC by Lemma \ref{lem_E[3]},
	$\mathcal{A}(E,\sigma_{p}\tau^{i}) \cong \mathcal{A}(E,\sigma_{q}\tau^{j})$
	if and only if
	there is $\varphi=\sigma_{s}\tau^{l} \in {\rm Aut}_{k}(\mathbb{P}^{2},E)$
	where $s \in E[3]$ and $l \in \mathbb{Z}_{d}$ such that the diagram
	\[\xymatrix{
		E \ar[r]^{\varphi} \ar[d]_{\sigma_{p}\tau^{i}} & E \ar[d]^{\sigma_{q}\tau^{j}} \\
		E \ar[r]^{\varphi} & E
	}\]
	commutes by Theorem \ref{gradediso}, that is,
	$$(\sigma_{q}\tau^{j})(\sigma_{s}\tau^{l})(\sigma_{p}\tau^{i})^{-1}=\sigma_{s}\tau^{l}.$$
	By Lemma \ref{formula}\,(\ref{formula1}),
	$(\sigma_{q}\tau^{j})(\sigma_{s}\tau^{l})(\sigma_{p}\tau^{i})^{-1}
	=\sigma_{q+\tau^{j}(s)-\tau^{l+j-i}(p)}\tau^{l+j-i}$, 
        so we have
	$q+\tau^{j}(s)-\tau^{l+j-i}(p)=s$ and $l+j-i=l$, 
        that is, $q=\tau^{l}(p)+s-\tau^{i}(s)$ and $i=j$.
	By the definition of $F_{i}$, $s-\tau^{i}(s) \in F_{i}$,
	so $\mathcal{A}(E,\sigma_{p}\tau^{i}) \cong \mathcal{A}(E,\sigma_{q}\tau^{j})$
	if and only if $i=j$ and $q=\tau^{l}(p)+r$ 
        where $r \in F_{i}$ and $l \in \mathbb{Z}_{d}$.
\end{proof}
By \cite{F}, we label the elements of $E_{\la}[3]$ by
\begin{align*}
&p_{0}:=o_{\la}:=(1:-1:0),\quad p_{1}:=(1:-\ep:0),\quad p_{2}:=(1:-\ep^{2}:0), \\
&p_{3}:=(1:0:-1),\quad p_{4}:=(1:0:-\ep),\quad p_{5}:=(1:0:-\ep^2), \\
&p_{6}:=(0:1:-1),\quad p_{7}:=(0:1:-\ep),\quad p_{8}:=(0:1:-\ep^2).
\end{align*}
We calculate 
$F_{\la,i}=\{ p_{l}-\tau_{\la}^{i}(p_{l}) \in E_{\la}\,|\,0 \leq l \leq 8 \}$ 
for each $i \in \mathbb{Z}_{d_{\la}}$.

\begin{lem}\label{genetor}
	\begin{enumerate}[{\rm (1)}]
		\item If $j(E_{\la}) \neq 0,12^3$, then
		$$F_{\la,i}=\begin{cases}
		\{ p_{0} \} &{\rm if}\,\,\, i=0,\\
		E_{\la}[3]  &\rm{otherwise}.
		\end{cases}$$
		
		\item If $\la=0$, then
		$$
		F_{\la,i}=\begin{cases}
		\{ p_{0} \} &{\rm if} \,\,\, i=0,\\
		\langle p_{1} \rangle=\{ p_{0},p_{1},p_{2} \} &{\rm if}\,\,\, i=2,4,\\
		E_{\la}[3] &\rm{otherwise}.
		\end{cases}
		$$
		
		\item If $\la=1+\sqrt{3}$, then
		$$
		F_{\la,i}=\begin{cases}
		\{ p_{0} \} &{\rm if}\,\,\, i=0,\\
		E_{\la}[3] &\rm{otherwise}.
		\end{cases}
		$$
	\end{enumerate}
\end{lem}
\begin{proof}
	By calculations.
\end{proof}
\begin{exm}\label{exm1}
	Fix $\la \in k$ such that $\la^{3} \neq 1$ 
        and $j(E_{\la}) \neq 0,12^3$ and
	let $p=(a:b:c) \in E_{\la}=\V(x^3+y^3+z^3-3\la xyz)$ 
        such that $abc \neq 0$.
	If $A=\mathcal{A}(E_{\la},\sigma_{p})$, then
	$$A=\kang/(ayz+bzy+cx^2, azx+bxz+cy^2, axy+byx+cz^2),$$
	and $A$ is a $3$-dimensional Sklyanin algebra.
	If $A'=\mathcal{A}(E_{\la},\sigma_{-p})$ where $-p=(b:a:c)$, then
	$$
        A'=\kang/(byz+azy+cx^2, bzx+axz+cy^2, bxy+ayx+cz^2),
        $$
	and $A'$ is also a $3$-dimensional Sklyanin algebra.
	If $A''=\mathcal{A}(E_{\la},\sigma_{p}\tau_{\la})$, then
	$$A''=\kang/(axz+bzy+cyx, azx+byz+cxy, ay^2+bx^2+cz^2).$$
	If $A'''=\mathcal{A}(E_{\la},\sigma_{p+p_{3}}\tau_{\la})$ 
        where $p_{3}:=(1:0:-1) \in E_{\la}[3]$,
	then
	$$A'''=\kang/(bxz+czy+ayx, bzx+cyz+axy, by^2+cx^2+az^2).$$
	By Theorem \ref{classify}, since $-p=\tau_{\la}(p)$ 
        and $p_{3} \in E_{\la}[3]=F_{\la,1}$,
	$A \cong A'$ and $A'' \cong A'''$ but no other pairs are isomorphic.
\end{exm}
\subsection{Classification up to graded Morita equivalence}
We recall that
$o:=\psi(o_{\la})$ and $\tau:=\Psi(\tau_{\la}) \in G={\rm Aut}_{k}(E,o)$.
Since $\tau$ is also a group automorphism of $(E,o)$, 
it follows that $\tau(E[3])=E[3]$.
\begin{lem}\label{E[3]}
	For $p \in E$ and $l \in \mathbb{Z}$,
	if $p-\tau^{l}(p) \in E[3]$, then $p-\tau^{nl}(p) \in E[3]$ 
        for any $n \in \mathbb{Z}$.
\end{lem}
\begin{proof}
	If $n=0$, then $p-\tau^{nl}(p)=p-p=o \in E[3]$.
	
	For any $n \ge 1$, we can write
	$$p-\tau^{nl}(p)=\sum_{i=0}^{n-1}\tau^{il}(p-\tau^{l}(p)).$$
	Since $p-\tau^{l}(p) \in E[3]$, $\tau^{il}(p-\tau^{l}(p)) \in E[3]$ 
        for $1 \leq i \leq n-1$,
	so $p-\tau^{nl}(p) \in E[3]$.
	
	If $n \leq -1$, then $p-\tau^{nl}(p)=-\tau^{nl}(p-\tau^{-nl}(p))$.
	Since $-n \ge 1$ and $p-\tau^{-nl}(p) \in E[3]$, 
        it follows that $p-\tau^{nl}(p) \in E[3]$ for any $n \leq -1$.
\end{proof}
\begin{thm}
\label{main4}
Let $p,q \in E \setminus E[3]$ and $i,j \in \mathbb{Z}_{d}$. Then
$
{\rm GrMod}\,\mathcal{A}(E,\sigma_{p}\tau^{i}) 
\cong {\rm GrMod}\,\mathcal{A}(E,\sigma_{q}\tau^{j})
$ if and only if
$
p-\tau^{j-i}(p) \in E[3]
$
and there exist $r \in E[3]$ and $l \in \mathbb{Z}_{d}$ such that
$
q=\tau^{l}(p)+r
$. 
\end{thm}
\noindent
\begin{proof}
Suppose that
$
{\rm GrMod}\,\mathcal{A}(E,\sigma_{p}\tau^{i}) 
\cong {\rm GrMod}\,\mathcal{A}(E,\sigma_{q}\tau^{j})
$.
Since $\mathcal{A}(E,\sigma_{p}\tau^{i})$ 
and $\mathcal{A}(E,\sigma_{q}\tau^{j})$
are geometric algebras of Type EC by Lemma \ref{lem_E[3]},
there exists a sequence $\{ \phi_{n} \}_{n \in \mathbb{Z}}$
of ${\rm Aut}_{k}(\mathbb{P}^{2},E)$ 
such that the diagram
	\[ \xymatrix{
		E \ar[r]^{\phi_{n}} \ar[d]_{\sigma_{p}\tau^{i}} 
                & E \ar[d]^{\sigma_{q}\tau^{j}} \\
		E \ar[r]^{\phi_{n+1}} & E
	}\]
	commutes for $n \in \mathbb{Z}$ 
        by Theorem \ref{gradedequi}.
	By Theorem \ref{PE} (3), there exist $r \in E[3]$ 
        and $l \in \mathbb{Z}_{d}$ such that $\phi_{0}=\sigma_{r}\tau^{l}$.
	Since the diagrams
	\[\xymatrix{
		E \ar[r]^{\phi_{-1}} \ar[d]_{\sigma_{p}\tau^{i}} 
                & E \ar[d]^{\sigma_{q}\tau^{j}} \\
		E \ar[r]^{\sigma_{r}\tau^{l}} & E
	}
	\xymatrix{
		E \ar[r]^{\sigma_{r}\tau^{l}} \ar[d]_{\sigma_{p}\tau^{i}} 
                & E \ar[d]^{\sigma_{q}\tau^{j}} \\
		E \ar[r]^{\phi_{1}} & E
	}\]
	commute,
        \begin{align*}
	\phi_{-1}
        &=(\sigma_{q}\tau^{j})^{-1}(\sigma_{r}\tau^{l})(\sigma_{p}\tau^{i}) \\
	&=\sigma_{\tau^{-j}(-q+r+\tau^{l}(p))}\tau^{l+i-j}
	\end{align*}
	and
	\begin{align*}
	\phi_{1}
        &=(\sigma_{q}\tau^{j})(\sigma_{r}\tau^{l})(\sigma_{p}\tau^{i})^{-1} \\
	&=\sigma_{q+\tau^{j}(r)-\tau^{l+j-i}(p)}\tau^{l+j-i}
	\end{align*}
	by Lemma \ref{formula}. 
        Since $\phi_{-1}, \phi_{1} \in {\rm Aut}_{k}(\mathbb{P}^{2},E)$,
	we have
	\begin{align*}
	\tau^{-j}(-q+r+\tau^{l}(p)) &\in E[3], \\
	q+\tau^{j}(r)-\tau^{l+j-i}(p) &\in E[3],
	\end{align*}
	that is,
	\begin{align*}
	s&:=-q+r+\tau^{l}(p)=\tau^{j}(\tau^{-j}(-q+r+\tau^{l}(p))) \in E[3], \\
	t&:=q+\tau^{j}(r)-\tau^{l+j-i}(p) \in E[3].
	\end{align*}
	By the first condition, 
        we have $q=\tau^{l}(p)+r-s$ where $r-s \in E[3]$.
	Since $s+t=r+\tau^{j}(r)+\tau^{l}(p)-\tau^{l+j-i}(p) \in E[3]$, 
        we have
	\begin{align*}
	p-\tau^{j-i}(p)
        &=\tau^{-l}(\tau^{l}(p)-\tau^{l+j-i}(p)) \\
	&=\tau^{-l}(s+t-r-\tau^{j}(r)) \in E[3]. 
	\end{align*}
	
	Conversely, suppose that
	$p-\tau^{j-i}(p) \in E[3]$ and $q=\tau^{l}(p)+r$ 
        where $r \in E[3]$ and $l \in \mathbb{Z}_{d}$.
	By Theorem \ref{classify}, we have
	$$
        \mathcal{A}(E,\sigma_{q}\tau^{j})
        =\mathcal{A}(E,\sigma_{\tau^{l}(p)+r}\tau^{j})
	=\mathcal{A}(E,\sigma_{\tau^{l}(p+\tau^{-l}(r))}\tau^{j}) 
        \cong \mathcal{A}(E,\sigma_{p+\tau^{-l}(r)}\tau^{j}).
        $$
	To show
	$$
        {\rm GrMod}\,\mathcal{A}(E,\sigma_{p}\tau^{i}) 
        \cong {\rm GrMod}\,\mathcal{A}(E,\sigma_{q}\tau^{j}),
        $$
	it is enough to show
	$$
        {\rm GrMod}\,\mathcal{A}(E,\sigma_{p}\tau^{i}) 
        \cong {\rm GrMod}\,\mathcal{A}(E,\sigma_{p+s}\tau^{j})
        $$
	where $s=\tau^{-l}(r) \in E[3]$.
	Since $p \in E \setminus E[3]$, 
        $p+s \in E \setminus E[3]$, so
	$\mathcal{A}(E,\sigma_{p}\tau^{i})$ 
        and $\mathcal{A}(E,\sigma_{p+s}\tau^{j})$
	are geometric algebras of Type EC by Lemma \ref{lem_E[3]}.
	We construct a sequence of automorphisms 
        $\{ \phi_{n} \}_{n \in \mathbb{Z}}$ of ${\rm Aut}_{k}\,(\mathbb{P}^{2},E)$.
	We set $\phi_{0}:={\rm id}_{E}$.
	For each $n \ge 1$, we define $\phi_{n}$ inductively as
	$$\phi_{n}:=\sigma_{r_{n}}\tau^{n(j-i)},$$
	where $r_{n}:=p-\tau^{n(j-i)}(p)+s+\tau^{j}(r_{n-1})$ and $r_{0}:=o$.
	For any $n \ge 0$, if $r_{n} \in E[3]$, then
	$r_{n+1}:=p-\tau^{(n+1)(j-i)}(p)+s+\tau^{j}(r_{n}) \in E[3]$ 
        by Lemma \ref{E[3]}.

	Next, for $n \leq -1$, we construct automorphisms 
        $\phi_{n} \in {\rm Aut}_{k}(\mathbb{P}^{2},E)$.
	For each $n \ge 1$, we define $\phi_{-n}$ inductively as
	$$\phi_{-n}:=\sigma_{r_{-n}}\tau^{-n(j-i)},$$
	where $r_{-n}:=\tau^{(n-1)i-nj}(p-\tau^{(n-1)(j-i)}(p)+\tau^{(n-1)(j-i)}(-s+r_{-(n-1)}))$ and $r_{0}:=o$.
	For any $n \ge 0$, if $r_{-n} \in E[3]$, 
        then $r_{-(n+1)}:=\tau^{ni-(n+1)(j-i)}(p-\tau^{n(j-i)}(p)+\tau^{n(j-i)}(-s+r_{-n})) \in E[3]$ by Lemma \ref{E[3]}.
	By this construction, 
        we have the sequence of automorphisms 
        $\{ \phi_{n} \}_{n \in \mathbb{Z}}$ of ${\rm Aut}_{k}(\mathbb{P}^{2},E)$
	such that the diagram
	\[ \xymatrix{ 
		E \ar[r]^{\phi_{n}} \ar[d]_{\sigma_{p}\tau^{i}} 
                & E \ar[d]^{\sigma_{p+s}\tau^{j}} \\
		E \ar[r]^{\phi_{n+1}} & E
	}\]
	commutes for each $n \in \mathbb{Z}$, 
        so 
	$
        {\rm GrMod}\,\mathcal{A}(E,\sigma_{p}\tau^{i}) 
        \cong {\rm GrMod}\,\mathcal{A}(E,\sigma_{p+s}\tau^{j})
        $
	by Theorem \ref{gradedequi}.
\end{proof}
\begin{exm}\label{exam2}
    We use the same graded algebras $A$ and $A''$ 
    as in Example \ref{exm1}
    so that $A \not\cong A''$.
    It follows from Theorem \ref{main4} that 
    ${\rm GrMod}\,A \cong {\rm GrMod}\,A''$ if and only if
    $2p=p-\tau_{\la}(p) \in E_{\la}[3]$ if and only if
    $p \in E_{\la}[6]$. 
    From \cite{F}, we see that $|E_{\la}[6]|=36$, 
    so $A$ and $A''=A^{\phi_{\la}}$ are rarely graded Morita equivalent 
    (see Remark \ref{ta}). 
\end{exm}
\section*{Acknowledgments}
The first author was supported by 
Grants-in-Aid for Young Scientific Research 18K13397 
Japan Society for the Promotion of Science.
The authors appreciate Sho Matsuzawa, Kim Gahee, Jame Eccels, Ryo Onozuka, 
Shinichi Hasegawa and Kosuke Shima 
for their helping to build and to check 
the tables in Theorem \ref{thm_main} and Theorem \ref{thm_main2}.
The authors also appreciate Shinnosuke Okawa 
for a useful comment on Remark \ref{Z-algs}. 
At most the authors would like to thank Izuru Mori
for his supervision on this work.

\end{document}